\documentclass[10pt,letterpaper]{amsart}
\usepackage[utf8]{inputenc}
\usepackage{amsmath}
\usepackage{amsfonts}
\usepackage{amssymb}
\usepackage{enumitem}
\usepackage{tikz}
\usepackage{MnPrime}
\usepackage{hyperref}
\usepackage[backend=biber,
	sorting=none,
    isbn=false,
    doi=false,
    url=false]{biblatex}

\newtheorem{lemma}{Lemma}
\newtheorem{theorem}[lemma]{Theorem}

\title[A few more trees]{A few more trees the chromatic symmetric function can distinguish}
\author{Jake Huryn}
\address{The Ohio State University.}
\email{huryn.5@osu.edu}
\date{}

\begin{document}

\begin{abstract}
A well-known open problem in graph theory asks whether Stanley's \textit{chromatic symmetric function}, a generalization of the chromatic polynomial of a graph, distinguishes between any two non-isomorphic trees. Previous work has proven the conjecture for a class of trees called \textit{spiders}. This paper generalizes the class of spiders to \textit{$n$-spiders}, where normal spiders correspond to $n=1$, and verifies the conjecture for $n=2$.
\end{abstract}

\maketitle

\section{The Chromatic Symmetric Function and its Coefficients}

Let $G=(V,E)$ be a graph, and let $\mathcal K(F)$ denote the set of connected components of the graph $(V,F)$ for any $F\subset E$. Then the \textit{chromatic symmetric function} of $G$ can be expressed as \cite[Theorem 2.5]{St} \begin{equation}\label{csf}
X_G=\sum_{F\subset E}{\left((-1)^{|F|}\prod_{K\in\mathcal K(F)}p_{|V(K)|}\right)}.
\end{equation}
This is the so-called \textit{subsets of edges} formulation of $X_G$, for reasons which will shortly be made clear. Here the symbols $p_k$ represent the power sum symmetric functions, but will be treated as formal commuting indeterminants. Up to the sign $(-1)^{|F|}$ this polynomial was introduced in \cite[Section 1.3]{CDL94}, motivated by problems in knot theory, and its relationship with Stanley's definition was observed in \cite[Theorem 6.1]{NW99}. The question of whether this invariant distinguishes all non-isomorphic trees has been studied for many specific classes of trees, for example in \cite{MMW} and \cite{LS}, and has been computationally verified for all trees with up to 29 vertices \cite{HJ}.

Recall that a \textit{tree} is a connected graph lacking cycles. The removal of any edge of a tree disconnects the graph into a disjoint union of two trees, and so the removal of $n$ edges produces a disjoint union of $n+1$ trees. A \textit{leaf} of a tree is a vertex of degree one.

Let $T=(V,E)$ be a tree. We will discuss subsets $F\subset E$ by how many edges are removed from $E$, which is one fewer than the number of connected components of $(V,F)$, and thus one fewer than the number of factors in the product \[\prod_{K\in\mathcal K(F)}p_{|V(K)|}.\] This means that no cancellation will occur in $X_T$, since if terms have opposite signs, they must have a different number of $p_k$ factors. (This is false if $T$ is not a tree, and indeed, general graphs are not distinguished by their chromatic symmetric functions; see \cite{MMW} for examples.) This lack of cancellation is important because it allows us to give, for trees, a coherent combinatorial interpretation of (\ref{csf}).

Moreover, each $F\subset E$ induces a partition $\lambda=(\lambda_1,\dots,\lambda_\ell)$ of $|V|$. The parts of $\lambda$ are the numbers $|V(K)|$ for all $K\in\mathcal K(F)$, since the connected components of $(V,F)$ partition $V$ itself. We see that in (1), the term in the sum corresponding to $F$ will be $(-1)^{|F|}p_{\lambda_1}\!\cdots p_{\lambda_\ell}$. Given a partition $\lambda$ of $|V|$, denote the absolute value of the coefficient on $p_{\lambda_1}\!\cdots p_{\lambda_\ell}$ in $X_T$ (now considering all subsets of $E$) by $c_\lambda(T)=c_{\lambda_2,\dots,\lambda_\ell}(T)$. Thus $c_\lambda(T)$ is the number of subsets of the edge set which induce the partition $\lambda$. We will write the subscripts in weakly decreasing order, omitting the largest part, and often simply write $c_\lambda$ or $c_{\lambda_2,\dots,\lambda_\ell}$ when working with a specific tree.

We have come to the combinatorial interpretation for $c_\lambda$ (and thus for (\ref{csf})), which represents the number of ways we can remove $\ell-1$ edges from $T$ (an \textit{$(\ell-1)$-cut} of $T$) to partition $T$, via the spanning subgraph, into connected components of order $\lambda_1,\dots,\lambda_\ell$ (Figure~\ref{fig1}). In particular, the reader is encouraged to check that $c_1$ is the number of leaves of $T$. Intuitively, $X_T$ contains all information about the sizes of the components we can get when we remove any subset of the tree's edges. Stanley's conjecture is exactly that this data is sufficient to completely determine any tree's isomorphism class.

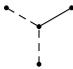
\begin{figure}
\begin{tikzpicture}[scale=1/2]
\draw[densely dashed] (0,0) -- (0,1) -- ({-sqrt(3)/2},3/2);
\draw (0,1) -- ({+sqrt(3)/2},3/2);
\draw[fill] (0,0) circle [radius=0.05];
\draw[fill] (0,1) circle [radius=0.05];
\draw[fill] ({-sqrt(3)/2},3/2) circle [radius=0.05];
\draw[fill] ({+sqrt(3)/2},3/2) circle [radius=0.05];
\end{tikzpicture}
\caption{Dashed lines represent deletions, i.e., edges not in the chosen subset. The figure shows a 2-cut corresponding to the product $p_1^2p_2$ and contributing to the coefficient $c_{1,1}$. Note that there are three distinct ways to make a 2-cut of this tree which yields the same partition, so $c_{1,1}=3$ and $X_T$ has a $-3p_1^2p_2$ term. Considering all subsets of edges, we find that $X_T=p_1^4-3p_1^2p_2+3p_1p_3-p_4$, where each term corresponds to successively larger subsets.\label{fig1}}
\end{figure}

The difficulty of Stanley's conjecture, in an informal sense, comes from the fact that the information contained in each coefficient of $X_T$ is nonlocal; for example, we can determine from $X_T$ the number of vertices, number of leaves, the degree sequence, and the path sequence \cite[Corollary 5]{MMW}. The problem is how we can piece together this information, along with the other information in the coefficients, to see the small-scale shapes of the tree and how these structures are connected.

This paper attempts to develop a framework which allows for this to be done, and uses this framework to prove that a particular infinite family of trees is distinguished by the chromatic symmetric function.

\section{Rooted Subtrees and $X_T$}

Of central importance in the following results is the use of \textit{rooted trees}, that is, trees with one vertex designated as the root. Let $r(n)$ be the number of rooted tree isomorphism classes having order $n$. We will name each rooted tree isomorphism class $R_{n,i}$ (giving the same name to any rooted tree in this class), where $n$ is the order of the rooted tree and $i\in\{1,\dots,r(n)\}$ is an indexing of all isomorphism classes of order $n$. In particular, let $R_{n,1}$ denote a path with the root at a leaf, and, for $n\geq2$, let $R_{n+1,2}$ be the same, with the addition of a single vertex appended to the second outward-most vertex from the root (Figure~\ref{fig2}).

\begin{figure}[h]
\begin{tikzpicture}[scale=1/2]
\draw[densely dashed] (-6,-1) -- (-6,0);
\draw[fill] (-6,0) circle [radius=0.08];
\draw[densely dashed] (-3,-1) -- (-3,0);
\draw (-3,0) -- (-3,1);
\draw[fill] (-3,0) circle [radius=0.08];
\draw[fill] (-3,1) circle [radius=0.05];
\draw[densely dashed] (0,-1) -- (0,0);
\draw (0,0) -- (0,1) -- (0,2);
\draw[fill] (0,0) circle [radius=0.08];
\draw[fill] (0,1) circle [radius=0.05];
\draw[fill] (0,2) circle [radius=0.05];
\draw[densely dashed] (3,-1) -- (3,0);
\draw ({3+sqrt(3)/2},1/2) -- (3,0) -- ({3-sqrt(3)/2},1/2);
\draw[fill] (3,0) circle [radius=0.08];
\draw[fill] ({3-sqrt(3)/2},1/2) circle [radius=0.05];
\draw[fill] ({3+sqrt(3)/2},1/2) circle [radius=0.05];
\draw[densely dashed] (6,-1) -- (6,0);
\draw (6,0) -- (6,1) -- ({6-sqrt(3)/2},3/2);
\draw (6,1) -- ({6+sqrt(3)/2},3/2);
\draw[fill] (6,0) circle [radius=0.08];
\draw[fill] (6,1) circle [radius=0.05];
\draw[fill] ({6-sqrt(3)/2},3/2) circle [radius=0.05];
\draw[fill] ({6+sqrt(3)/2},3/2) circle [radius=0.05];
\end{tikzpicture}
\caption{Rooted subtrees $R_{1,1}$, $R_{2,1}$, $R_{3,1}$, $R_{3,2}$, and $R_{4,2}$, respectively, with the roots enlarged. The dashed lines are the unique edges connecting the rooted subtrees to the rest of the tree.\label{fig2}}
\end{figure}
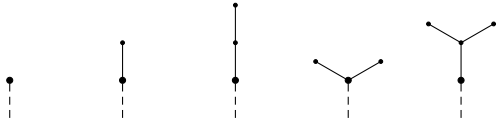

We also define $c_\lambda(R)$ for any rooted tree $R$. Given a partition $\lambda$, define $c_\lambda(R)$ as the number of ways to cut $R$ into $|\lambda|+1$ connected components such that the orders of those components not containing the root correspond to the parts of $\lambda$. In this case, no parts of $\lambda$ will be omitted from the subscript of $c_\lambda$, since the connected component containing the root is already omitted from $\lambda$. For example, the root of $R$ may also be a leaf, but it will not contribute to $c_1(R)$. In particular, we have $c_1(R_{n,1})=1$ and $c_1(R_{n,2})=2$ for all $n$.

In practice, these rooted trees will be employed as \textit{rooted subtrees} of a tree $T$. Suppose we remove an edge from $T$, splitting $T$ into two connected components. Then each component forms a rooted subtree of $T$, in which the root is the vertex incident to the edge we removed. This process produces every rooted subtree of $T$. Let the number of rooted subtrees of $T$ isomorphic to $R_{n,i}$ be denoted $\rho_{n,i}(T)$, or simply $\rho_{n,i}$. The method of this paper hinges upon the calculation of these numbers $\rho_{n,i}$, which essentially represent the small-scale shapes of the tree, from the coefficients of $X_T$.

Let $T$ have order $d$. It follows from our definitions that we have the equations \[c_n=\sum_{i=1}^{r(n)}\rho_{n,i}=\sum_{i=1}^{r(d-n)}\rho_{d-n,i}\] if $n\neq d/2$, and \[c_{d/2}=\frac{1}{2}\sum_{i=1}^{r(d/2)}\rho_{d/2,i}.\] The following lemma is a similar but more complicated equation for $c_{n,1}$. The reader is invited to check that $c_{1,1}=\binom{c_1}2+c_2$, but this does not give us any more information about $T$.

\begin{lemma}\label{lemma1}
Let $T=(V,E)$ be a tree of order $d$, and let $T_{n,i}$ denote the tree obtained from $R_{n,i}$ by forgetting the distinction of the root. Then if $2\leq n<(d-1)/2$,
\begin{equation}\label{cn1}
c_{n,1}=\sum_{i=1}^{r(n)}(c_1-c_1(R_{n,i}))\rho_{n,i}+\sum_{j=1}^{r(n+1)}c_1(T_{n+1,j})\rho_{n+1,j}.
\end{equation}
Also, if $d\geq6$, then
\begin{equation}\label{c1111}
c_{1,1,1,1}=\binom{c_1}4+\binom{c_1-1}2c_2+\binom{c_2}2+(c_1-1)\rho_{3,1}+(c_1-2)\rho_{3,2}+c_4.
\end{equation}
\end{lemma}

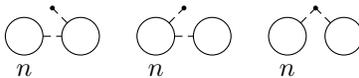
\begin{figure}[h]
\begin{tikzpicture}[scale=1/2]
\draw (0,0) circle [radius=1/2];
\node[below] at (0,-1/2) {$n$};
\draw (3/2,0) circle [radius=1/2];
\draw[fill] (3/4,3/4) circle [radius=0.05];
\draw[densely dashed] ({3/2-sqrt(2)/4},{sqrt(2)/4}) -- (3/4,3/4);
\draw[densely dashed] (1/2,0) -- (1,0);

\draw (7/2,0) circle [radius=1/2];
\node[below] at (7/2,-1/2) {$n$};
\draw (5,0) circle [radius=1/2];
\draw[fill] (17/4,3/4) circle [radius=0.05];
\draw[densely dashed] ({7/2+sqrt(2)/4},{sqrt(2)/4}) -- (17/4,3/4);
\draw[densely dashed] (4,0) -- (9/2,0);

\draw (7,0) circle [radius=1/2];
\node[below] at (7,-1/2) {$n$};
\draw (17/2,0) circle [radius=1/2];
\draw[fill] (31/4,3/4) circle [radius=0.05];
\draw[densely dashed] ({7+sqrt(2)/4},+{sqrt(2)}/4) -- (31/4,3/4);
\draw[densely dashed] ({17/2-sqrt(2)/4},{sqrt(2)/4}) -- (31/4,3/4);
\end{tikzpicture}
\caption{The three ways of joining the connected components of a partition contributing to $c_{n,1}$. The first picture corresponds to the first term of (\ref{cn1}). The labeled components have order $n$.\label{fig3}}
\end{figure}

\begin{proof}The coefficient $c_{n,1}$ tells us the number of ways we can cut $T$ in two places to get a connected component of order one and another of order $n$, and of course a third of order $d-n-1$. There are three ways of joining these components with two edges as illustrated by Figure~\ref{fig3}.

The derivation of (\ref{cn1}) should then be clear from the figure. In particular, observe that the second picture requires that the isolated vertex not originally be the root of the rooted subtree of order $n+1$, while the third requires exactly the opposite. Thus, both terms combined negate the distinction of the root.

Now let $d\geq6$. From Figure~\ref{fig4} we can see that (\ref{c1111}) holds; we require $d\geq6$ so that the non-dashed circles drawn in Figure~\ref{fig4} are distinguished from a vertex.\end{proof}

Thus the coefficients $c_{1,1,1,1}$ and $c_{3}=\rho_{3,1}+\rho_{3,2}$ (or $c_3=\frac{1}{2}\rho_{3,1}+\frac{1}{2}\rho_{3,2}$ if $d=6$) form a system of equations which allows us to solve for $\rho_{3,1}$ and $\rho_{3,2}$.

\begin{figure}[h]
\begin{tikzpicture}[scale=0.5]
\draw (0,0) circle [radius=0.5];
\draw[fill] (1,0) circle [radius=0.05];
\draw[fill] (0,1) circle [radius=0.05];
\draw[fill] (-1,0) circle [radius=0.05];
\draw[fill] (0,-1) circle [radius=0.05];
\draw[densely dashed] (1,0) -- (0.5,0);
\draw[densely dashed] (0,1) -- (0,0.5);
\draw[densely dashed] (-1,0) -- (-0.5,0);
\draw[densely dashed] (0,-1) -- (0,-0.5);

\draw (3,0) circle [radius=0.5];
\draw[fill] (4,0) circle [radius=0.05];
\draw[fill] (3,1) circle [radius=0.05];
\draw[fill] (3,1.5) circle [radius=0.05];
\draw[fill] (3,-1) circle [radius=0.05];
\draw[densely dashed] (4,0) -- (3.5,0);
\draw[densely dashed] (3,1.5) -- (3,0.5);
\draw[densely dashed] (3,-1) -- (3,-0.5);

\draw (6,0) circle [radius=0.5];
\draw[fill] (6,1) circle [radius=0.05];
\draw[fill] (6,1.5) circle [radius=0.05];
\draw[fill] (6,-1) circle [radius=0.05];
\draw[fill] (6,-1.5) circle [radius=0.05];
\draw[densely dashed] (6,1.5) -- (6,0.5);
\draw[densely dashed] (6,-1.5) -- (6,-0.5);

\draw (8.5,0) circle [radius=0.5];
\draw[fill] (8.5,1) circle [radius=0.05];
\draw[fill] (8.5,1.5) circle [radius=0.05];
\draw[fill] (8.5,2) circle [radius=0.05];
\draw[fill] (8.5,-1) circle [radius=0.05];
\draw[densely dashed] (8.5,2) -- (8.5,0.5);
\draw[densely dashed] (8.5,-1) -- (8.5,-0.5);

\draw (11,0) circle [radius=0.5];
\draw[fill] (11,1) circle [radius=0.05];
\draw[fill] ({11-sqrt(3)/4},1.25) circle [radius=0.05];
\draw[fill] ({11+sqrt(3)/4},1.25) circle [radius=0.05];
\draw[fill] (11,-1) circle [radius=0.05];
\draw[densely dashed] ({11-sqrt(3)/4},1.25) -- (11,1) -- (11,0.5);
\draw[densely dashed] ({11+sqrt(3)/4},1.25) -- (11,1);
\draw[densely dashed] (11,-1) -- (11,-0.5);

\draw (13.5,0) circle [radius=0.5];
\draw[densely dashed] (13.5,1.5) circle [radius=0.5];
\draw[densely dashed] (13.5,0.5) -- (13.5,1);
\node[right] at (14,1.5) {4};
\end{tikzpicture}
\caption{The ways of joining the connected components of the partition contributing to $c_{1,1,1,1}$. The pictures correspond to the respective terms in (\ref{c1111}).\label{fig4}}
\end{figure}
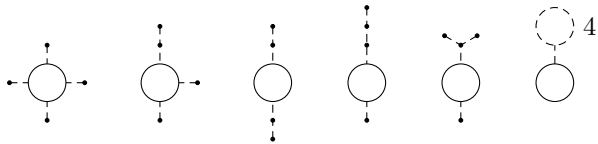

\section{Distinguishing Spiders and their Generalizations}

A \textit{spider} is a tree with exactly one vertex of degree at least three. That vertex is called the \textit{torso} of the spider, and the other vertices form paths extending from the torso called \textit{legs}. It is shown in \cite[Section 3]{MMW} that the chromatic symmetric function completely distinguishes spiders.

Call a tree a \textit{2-spider} if it is a modification of a spider in which any leg may be appended with a single vertex joined to the second outward-most vertex (with respect to the torso) of that leg. Call such modified legs the \textit{2-legs}. The reasoning behind these names is that one can think of these modified legs as being copies of the rooted subtree $R_{n,2}$, while normal legs are copies of $R_{n,1}$. We may thus call normal spiders \textit{1-spiders} and normal legs \textit{1-legs}. The order of a leg is defined to not include the root, so that the order of a (2-)spider is one more than the sum of the orders of each leg (Figure~\ref{fig5}).

\begin{figure}[h]
\begin{tikzpicture}[scale=0.5]
\draw[fill] (0,0) circle [radius=0.08];
\draw[fill] (1,0) circle [radius=0.05];
\draw[fill] (2,0) circle [radius=0.05];
\draw[fill] (2.5,{sqrt(3)/2}) circle [radius=0.05];
\draw[fill] (2.5,{-sqrt(3)/2}) circle [radius=0.05];
\draw[fill] (-2.5,{sqrt(3)/2}) circle [radius=0.05];
\draw[fill] (-2.5,{-sqrt(3)/2}) circle [radius=0.05];
\draw[fill] (0,1) circle [radius=0.05];
\draw[fill] (-1,0) circle [radius=0.05];
\draw[fill] (-2,0) circle [radius=0.05];
\draw[fill] (0,-1) circle [radius=0.05];
\draw (2.5,{sqrt(3)/2}) -- (2,0) -- (2.5,{-sqrt(3)/2});
\draw (-2.5,{sqrt(3)/2}) -- (-2,0) -- (-2.5,{-sqrt(3)/2});
\draw (0,1) -- (0,-1);
\draw (-2,0) -- (2,0);
\end{tikzpicture}
\caption{A 2-spider with two 2-legs, satisfying (ii) of Theorem~\ref{theorem2}, described by $(\lambda;\mu)=(1,1;4,4)$ and $(\lambda^*;\mu^*)=(2;2,2,2,2)$. The torso is enlarged.\label{fig5}}
\end{figure}
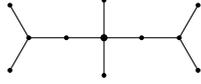

We may describe these 2-spiders up to isomorphism by a pair of positive integer sequences $(\lambda;\mu)=(\lambda_1,\dots,\lambda_\ell;\mu_1,\dots,\mu_m)$, where $\lambda$ lists the orders of the 1-legs and $\mu$ lists the orders of the 2-legs. With this in mind, we can now state and prove the main result.

\begin{theorem}\label{theorem2}
If $T_1$ and $T_2$ are non-isomorphic 2-spiders, then $X_{T_1}\neq X_{T_2}$.
\end{theorem}

\begin{proof}Let $T$ be a 2-spider with order $d$; the smallest 2-spider which is not also a 1-spider has order 6, so assume $d\geq6$. We will show that $T$, knowing that it is a 2-spider, can be reconstructed from $X_T$.

We would like to determine, for $1\leq n\leq d/2$, the number $\lambda_n^*$ of 1-legs with order at least $n$ and the number $\mu_n^*$ of 2-legs with order at least $n$ (although necessarily $\mu_1^*=\mu_2^*=\mu_3^*$, since any 2-leg must have order at least 3). If we consider $\lambda$ and $\mu$ to be partitions, then $\lambda^*$ and $\mu^*$ are their respective conjugate partitions, so with this information we can recover $T$.

The proof is separated based on the following (neither mutually exclusive nor exhaustive) possibilities:
\begin{enumerate}
\item $T$ has only three legs, two of order one.
\item If a leg of $T$ has order $n$, then $n\leq d/2$.
\end{enumerate}

First, suppose $T$ satisfies (i). We find that either $\rho_{3,2}=1$ and $c_1=3$, or $\rho_{3,2}=2$ and $c_1=4$. This is sufficient to distinguish this case, and we also know that the number of 2-legs of $T$ must be $\rho_{3,2}-1$. From this it is easy to reconstruct $T$ since we know $d$.

Now suppose $T$ satisfies (ii) but not (i). Suppose we cut an edge of some leg $L$ to split $T$ into two rooted subtrees. Then the connected component containing the torso will not be of the form $R_{n,1}$ or $R_{n,2}$ for any $n$ (note that this is false if $T$ satisfies (i)). However, the other connected component will be of the form $R_{n,1}$ if $L$ is a 1-leg or $R_{n,2}$ if $L$ is a 2-leg, and $L$ must have order at least $n$. Thus if $n\geq2$, $\lambda_n^*=\rho_{n,1}$ and $\mu_n^*=\rho_{n,2}$ (Figure~\ref{fig6}). However, if $n=1$ then $\rho_{1,1}$ also counts both leaves on each 2-leg, so to correct for this double counting we instead use \[\lambda_1^*=\rho_{1,1}-2\rho_{3,2}.\]

\begin{figure}[b]
\begin{tikzpicture}[scale=0.5]
\draw[fill] (0,0) circle [radius=0.08];
\draw[fill] (1,0) circle [radius=0.05];
\draw[fill] (2,0) circle [radius=0.05];
\draw[fill] (2.5,{sqrt(3)/2}) circle [radius=0.05];
\draw[fill] (2.5,{-sqrt(3)/2}) circle [radius=0.05];
\draw[fill] (-2.5,{sqrt(3)/2}) circle [radius=0.05];
\draw[fill] (-2.5,{-sqrt(3)/2}) circle [radius=0.05];
\draw[fill] (0,1) circle [radius=0.05];
\draw[fill] (-1,0) circle [radius=0.05];
\draw[fill] (-2,0) circle [radius=0.05];
\draw[fill] (0,-1) circle [radius=0.05];
\draw (2.5,{sqrt(3)/2}) -- (2,0) -- (2.5,{-sqrt(3)/2});
\draw (-2.5,{sqrt(3)/2}) -- (-2,0) -- (-2.5,{-sqrt(3)/2});
\draw (0,1) -- (0,-1);
\draw (-2,0) -- (-1, 0);
\draw[densely dashed] (-1,0) -- (0,0);
\draw (0,0) -- (2,0);
\end{tikzpicture}
\caption{The dashed edge represents the edge we cut. Note that we have split the tree into two subtrees, one of which is of the form $R_{4,2}$, and the other, containing the torso, which is not of the form $R_{n,1}$ or $R_{n,2}$ since the tree does not satisfy (i). Thus the leg we cut contributes to $\rho_{4,2}$, and of course this means it must be a 2-leg with order at least 4. Note that we can also cut this same leg on different edges to contribute to $\rho_{3,2}$ and twice to $\rho_{1,1}$. \label{fig6}}
\end{figure}
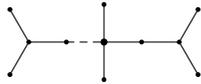

We know $\rho_{1,1}=c_1$ and $\rho_{2,1}=c_2$, and moreover we know $\rho_{3,1}$ and $\rho_{3,2}$ in terms of the coefficients $c_\lambda$ from (\ref{c1111}) and $c_3=\rho_{3,1}+\rho_{3,2}$. Now let $4\leq n<d/2$. It follows from the length restriction (ii) that any rooted subtree of $T$ containing the torso must have order at least $d/2$. Thus since any rooted subtree not containing the torso is either of the form $R_{n,1}$ or $R_{n,2}$, we know that if $i$ is not 1 or 2, then $\rho_{n,i}=0$. Then \[c_n=\rho_{n,1}+\rho_{n,2},\] and, by Lemma~\ref{lemma1}, \[c_{n-1,1}=(c_1-1)\rho_{n-1,1}+(c_1-2)\rho_{n-1,2}+2\rho_{n,1}+3\rho_{n,2}.\] Thus we may solve inductively for all such $\rho_{n,1}$ and $\rho_{n,2}$ with $n=3$ as the base case.

Now suppose $n=d/2$; there can be at most one leg of this order. From (\ref{cn1}) we can tell if there is a leg of order $n$ by examining $c_{n-1,1}$, since if there were not, the second sum in the equation would be zero. Note then that exactly one of $\rho_{n-1,1}$ and $\rho_{n-1,2}$ can be nonzero, since a 2-spider cannot have a leg of order $d/2-1$ if it also has one of order $d/2$. Thus we know the type of that largest leg. 

Notice that in case (ii), $c_i\geq c_j$ if $3\leq i<j<d/2$ since \[c_i=\rho_{i,1}+\rho_{i,2}=\lambda_i^*+\mu_i^*\geq\lambda_j^*+\mu_j^*=c_j.\]

Finally, suppose $T$ satisfies none of the cases. The smallest such 2-spider has order 9, so suppose $d\geq9$. Let the length of the (unique) largest leg have order $k>d/2$. A quick calculation verifies that we cannot also have a leg of order $d-k-1$ or larger, and that \[3\leq d-k-1<\frac{d}{2}-1.\] Thus $c_{d-k-1}=1$, but $c_{d-k}=2$, which distinguishes this case from the second case (Figure~\ref{fig7}). Moreover, we only need to find $\lambda_n^*$ and $\mu_n^*$ up to $n=d-k-1$, which will give the types of all legs, and their respective lengths, excluding the length of the longest leg. To do this, we simply solve for $\rho_{3,1}$ and $\rho_{3,2}$ and do the same process as in the second case up to $c_{d-k-2,1}$. Then we can find the length of the longest leg, thus reconstructing $T$, by looking at $d$.\end{proof}

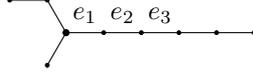
\begin{figure}
\begin{tikzpicture}[scale=0.5]
\draw[fill] (0,0) circle [radius=0.08];
\draw[fill] (1,0) circle [radius=0.05];
\draw[fill] (2,0) circle [radius=0.05];
\draw[fill] (3,0) circle [radius=0.05];
\draw[fill] (4,0) circle [radius=0.05];
\draw[fill] (5,0) circle [radius=0.05];
\draw[fill]  (-1.5,{sqrt(3)/2}) circle [radius=0.05];
\draw[fill] (-0.5,{sqrt(3)/2}) circle [radius=0.05];
\draw[fill] (-0.5,{-sqrt(3)/2}) circle [radius=0.05];
\draw (-1.5,{sqrt(3)/2}) -- (-0.5,{sqrt(3)/2}) -- (0,0) -- (-0.5,{-sqrt(3)/2});
\draw (0,0) -- (5,0);
\draw (0.5,0) node[above] {$e_1$};
\draw (1.5,0) node[above] {$e_2$};
\draw (2.5,0) node[above] {$e_3$};
\end{tikzpicture}
\caption{A 2-spider which satisfies neither (i) nor (ii), with torso enlarged; here $d=9$ and $k=5$. The edge $e_3$ is the only one we can cut to get a connected component of order 3, but both $e_1$ and $e_2$ give connected components of order 4, that is, $c_3=1$ but $c_4=2$.\label{fig7}}
\end{figure}

It turns out that solving Stanley's conjecture can be done by simply solving for all $\rho_{n,i}$ in terms of the coefficients of $X_T$. In Section 2 we did this for $n=3$, by examining $c_{3}$ and $c_{1,1,1,1}$, which by itself proves Stanley's conjecture for all trees of order at most 7. This avoids the unpleasant case-checking which allows Theorem~\ref{theorem2} to verify Stanley's conjecture for an infinite family. While there are some interesting heuristics for determining which $c_{\lambda}$ can yield which coefficients on the desired $\rho_{n,i}$ in its expansion, even for $n=4$ this is very difficult to do.

Even worse, if we were attempting to distinguish trees of order 20, we would need to solve for each variable $\rho_{10,i}$, of which there would be more than coefficients of $X_T$! What this seems to imply, then, is that there must be some “hidden” information about the values $\rho_{n,i}$ can take, as a consequence of the definition of a tree, that is not expressed directly in the chromatic symmetric function.

\subsection*{Acknowledgements.} This paper is the result of The Ohio State University's Knots and Graphs undergraduate research program, which can be found at \[\text{\url{http://www.math.ohio-state.edu/~chmutov/wor-gr-su18/wor-gr.htm},}\] and I am grateful to the OSU Honors Program Research Fund for financial support. I would like to thank my wonderful mentor Sergei Chmutov for his support and his knowledge of the topic, and am also indebted to Eric Fawcett, Rushil Raghavan, Ishaan Shah, and the other participants for our fruitful (and also not-so-fruitful) discussions throughout the summer.

\printbibliography

\end{document}